\documentclass[11pt]{article}

\usepackage{amsmath,amsthm,amssymb,url}
\usepackage{tikz}
\usepackage{tkz-graph}
\usepackage{authblk}
\usepackage[top=1.25in, bottom=1.5in, left=1.5in, right=1.5in]{geometry}

\newcommand{\eff}{{\ensuremath{\mathbb{F}}}} 
\newcommand{\A}{\ensuremath{\mathcal{A}}}
\newcommand{\zed}{{\ensuremath{\mathbb{Z}}}}

\newtheorem{theorem}{Theorem}[section]

\newtheorem{lemma}[theorem]{Lemma}
\newtheorem{corollary}[theorem]{Corollary}

\theoremstyle{definition}

\newtheorem{example}[theorem]{Example}
\newtheorem{definition}[theorem]{Definition}
\newtheorem{remark}[theorem]{Remark}

\title{Circular external difference families, graceful labellings and cyclotomy}
\author[1]{Maura B.\ Paterson}
\author[2,3]{Douglas R.\ Stinson\thanks{D.R.\ Stinson's research is supported by  NSERC discovery grant RGPIN-03882.}}
\affil[1]{School of Computing and Mathematical Sciences, Birkbeck, University of London, Malet St, London WC1E 7HX, UK}
\affil[2]{David R.\ Cheriton School of Computer Science\\University of Waterloo\\ Waterloo ON, N2L 3G1\\Canada}
\affil[3]{School of Mathematics and Statistics\\
Carleton University\\
Ottawa, Ontario, K1S 5B6, Canada}

\begin{document}
\maketitle

\begin{abstract}
(Strong) circular external difference families (which we denote as CEDFs and SCEDFs) can be used to construct nonmalleable threshold schemes.
They are a variation of (strong)  external difference families, which have been extensively studied in recent years. We provide a variety of constructions for CEDFs based on graceful labellings ($\alpha$-valuations) of lexicographic products  $C_n \boldsymbol{\cdot} K_{\ell}^c$, where $C_n$ denotes a cycle of length $n$. SCEDFs having more than two subsets do not exist. However, we can construct close approximations (more specifically, certain types of circular algebraic manipulation detection (AMD) codes) using the theory of cyclotomic numbers in finite fields. 
\end{abstract}

\section{Introduction}
\label{intro.sec}

Veitch and Stinson \cite{VS} introduced (strong) circular external difference families (denoted as CEDFs and SCEDFs) as a tool to construct nonmalleable threshold schemes. CEDFs and SCEDFs can be viewed as a variation of (strong)  external difference families (EDFs and SEDFs), which have been extensively studied in recent years. In this paper, we investigate the mathematical properties of  CEDFs and SCEDFs and provide several constructions. Many classes of CEDFs can be obtained using graceful labellings, in particular, $\alpha$-valuations of lexicographic products $C_n \boldsymbol{\cdot} K_{\ell}^c$, where $C_n$ is a cycle of length $n$.  

Wu, Yang and Feng \cite{feng} showed that SCEDFs on more than two subsets do not exist. However, we can construct close approximations (more specifically, certain types of circular algebraic manipulation detection (AMD) codes) using the theory of cyclotomic classes in finite fields.

We begin by presenting in Section \ref{back.sec} the cryptographic motivation for the mathematical problems we study in this paper. 
Following this discussion, in Section \ref{defn.sec}, we give formal definitions of the actual combinatorial structures that we are considering, as well as some background and previous results.

Section \ref{CEDF.sec} is devoted to CEDFs. In particular, we investigate $(m \ell^2 + 1,m,\ell; 1)$-CEDFs in detail (for the definition, see Definition \ref{CEDF.def}). Some of our main constructions make use of graceful labellings ($\alpha$-valuations) of
lexicographic products  $C_n \boldsymbol{\cdot} K_{\ell}^c$, where $C_n$ denotes a cycle of length $n$.

Section \ref{strong.sec} presents our results on SCEDFs and circular AMD codes (which can be viewed as approximations to SCEDFs).
In particular, we use known results on cyclotomic numbers to find good circular AMD codes on $m$ sets, for $m = 3,4$. The theory we develop extends to larger values of $m$ as well. Also, using a notion of ``tilings,'' we provide a structural proof that SCEDFs cannot contain certain SEDFs (these tiling results are subsumed by the general nonexistence result from \cite{feng}, but the proof technique might still be of interest). 

\subsection{Motivation}
\label{back.sec}

Various types of difference sets and difference families have had a long history in combinatorics. Applications of these objects to the construction of robust threshold schemes has been an active area of research in cryptography for almost 30 years, beginning with Ogata and Kurosawa in 1996 \cite{OK96}. Cramer {\it et al.} \cite{CDFPW1} introduced 
\emph{algebraic manipulation detection codes} (AMD codes) in 2008, motivated by applications to robust secret sharing and other cryptographic problems. AMD codes include many types of difference sets and difference families as special cases. In 2016, Paterson and Stinson \cite{PS16} established that \emph{R-optimal AMD codes} corresponded to external difference families (EDFs) and strong  external difference families (SEDFs). EDFs had been previously studied by various authors, but strong EDFs opened up a new area of research. Several papers have studied SEDFs since the publication of \cite{PS16}, including 
\cite{BJWZ,HucJefNep,HuPa,JeLi,LeLiPr,LePr,MS17,WYFF,feng}.

We give an informal description of AMD codes now. Roughly speaking, we have a finite set of possible \emph{sources}. Each source has one of more possible \emph{encodings}. A source $s$ is chosen uniformly at random from the set of possible sources and then it is encoded as an element $g$ of an additive abelian group $G$ (the encoding may be randomized). In an AMD code, the adversary chooses a nonzero element $\Delta \in G$ (without knowing $s$ or $g$). The adversary ``wins'' if $g+\Delta$ is a valid encoding of a source $s'\neq s$.  On the other hand, in a strong AMD code, the adversary is given the  source $s$ (but not the encoded source $g$) before they choose $\Delta \in G$.  The winning condition is the same as before.

The basic idea can be illustrated by considering a difference set, e.g., the $(13,4,1)$-difference set $D = \{0,1,3,9\} \subseteq \zed_{13}$. Suppose we have four equiprobable sources, say $s_1,s_2,s_3,s_4$, which are encoded (deterministically) as $0,1,3,9$, resp. So an encoded source is an element of $D$, each occurring with probability $1/4$. For any $\Delta \in \zed_{13} \setminus \{0\}$, the probability that $g+\Delta \in D$ is $1/4$. This follows from the fact that every nonzero $\Delta$ occurs exactly once as a difference of two distinct elements in $D$.

In the context of a threshold scheme, an AMD code can be use to add \emph{robustness}. The basic method is to first encode the secret and then create shares for the encoded secret using a Shamir threshold scheme (see \cite{ACM:Sha79}), for example. 

Veitch and Stinson recently studied a new variation of \emph{non-malleable threshold schemes} in \cite{VS}. Interesting examples of these threshold schemes can be obtained from circular external difference families (CEDFs) and strong circular external difference families (SCEDFs), which were defined in \cite{VS}. In this paper, we mainly concentrate on constructions of these combinatorial structures and corresponding circular AMD codes.

Here is a brief cryptographic motivation for circular AMD codes and circular EDFs in the context of non-malleable secret sharing. Analogous to the original definition of an AMD code, we can consider a circular AMD code. 
Suppose the group $G$ is a cyclic group. The idea is that the adversary is attempting to modify the encoding so it that it is a valid encoding of a source that has a specified, predetermined relation with the original source $s$, namely the modified source should be $s+1$ in this case. The only difference in the definition of the AMD code is that the winning condition for the adversary is now that $g+\Delta$ is a valid encoding of the source $s+1$ (where $s$ is the original source). 
Analogous to the case with AMD codes and CEDFs, the ``optimal'' examples of circular AMD codes are circular EDFs, as was shown in \cite{VS}.

\subsection{Definitions and Background}
\label{defn.sec}

We now  recall some definitions given in \cite{PS16,VS}. For two disjoint subsets $A,B$ of an additive abelian group $G$, define the multiset $\mathcal{D}(B,A)$ as follows:
\[ \mathcal{D}(B,A) = \{ y - x: y \in B, x \in A \}.  \]

\begin{definition}[Circular weak AMD code]
Let $\mathcal{G}$ be an additive abelian group of order $n$ and let $\mathcal{A} = \{{A}_0,\dots,{A}_{m-1}\}$ be $m$ pairwise disjoint $\ell$-subsets of $\mathcal{G}$. Let $0 < \epsilon < 1$. Then $(\mathcal{G},\mathcal{A})$ is an \emph{$\epsilon$-secure circular $(n,m,\ell)$-AMD code} if an adversary cannot win the following \emph{circular AMD game} with  probability greater than $\epsilon$.
\begin{enumerate}
\item The adversary chooses a value $\Delta \in \mathcal{G} \setminus \{0\}$. \vspace{-.08in}
\item The \emph{source} $i \in \{0,\dots,m-1\}$ is chosen uniformly at random. \vspace{-.08in}
\item The source is encoded by choosing $g$ uniformly at random from ${A}_i$. \vspace{-.08in}
\item The adversary wins if and only if $g + \Delta \in A_j$ where $j = i + 1 \bmod m$. \vspace{-.12in}
\end{enumerate}
\end{definition}

\medskip

\begin{definition}[Circular external difference family (CEDF)]
\label{CEDF.def}
Let $G$ be an additive abelian group of order $n$. Suppose $m \geq 2$. 
An $(n, m, \ell; \lambda)$-circular external difference family (or $(n, m, \ell; \lambda)$-CEDF) is a set of $m$ disjoint $\ell$-subsets of $G$, say $\mathcal{A} = \{A_0,\dots,A_{m-1}\}$, such that the following multiset equation holds:
\[
\bigcup_{j=0}^{m-1} \mathcal{D}(A_{j+1\bmod m}, A_j) = \lambda (G \setminus \{0\}).
\]
We observe that $m \ell^2 = \lambda (n-1)$ if an $(n, m, \ell; \lambda)$-CEDF exists. When $\lambda = 1$, we have $n = m \ell^2 + 1$.
\end{definition}

\medskip

It is not hard to see that an $\epsilon$-secure weak circular $(n,m,\ell)$-AMD code 
has $\epsilon \geq {\ell}/({n-1})$. If $\epsilon = {\ell}/({n-1})$, the code is \emph{R-optimal}.
Veitch and Stinson proved the following result in 
\cite{VS}.

\begin{theorem}
\label{VSthm}
An R-optimal circular weak $(n,m,\ell)$-AMD code is equivalent to an $(n,m,\ell;\lambda)$-CEDF.
\end{theorem}

We give an example of a CEDF from \cite{VS}.
\begin{example} 
\label{ex1742}
The following four sets of size $2$ form a $(17,4,2,1)$-CEDF in $\zed_{17}$:
\[\A = (
\{1 ,16\},
\{9 ,8\},
\{13 ,4\},
\{15 ,2)\}).\]
This yields an (R-optimal) $1/8$-secure $(17,4,2)$-AMD code. $\blacksquare$
\end{example}

We recall a couple of theorems on CEDFs that can be found in \cite{VS}. 

\begin{theorem}
\cite{VS}
\label{thm2}
Suppose that $q = m\ell^2 +1$ is a prime power and $\alpha$ is a primitive element of ${\eff_q}$. 
Define $\beta = \alpha^{\ell}$
and let $H$ be the subgroup of ${\eff_q}^*$ of order $\ell m$ generated by $\beta$.
Then there is a $(q,m,\ell;1)$-CEDF in $\eff_q$ if 
\[\{ \beta - 1 , \beta^{m+1}-1, \dots , \beta^{(\ell-1)m+1} -1 \}\] is
a set of coset representatives of $H$ in ${\eff_q}^*$.
\end{theorem}

The following theorem is obtained by taking $\ell = 2$ in Theorem \ref{thm2}. 

\begin{theorem}
\cite{VS}
\label{thm1}
Suppose that $q = 4m+1$ is a prime power and suppose there exists a primitive element $\alpha \in \eff_q$ such that $\alpha^4-1$ is a quadratic non-residue. Then there is a $(q,m,2;1)$-CEDF in $\eff_q$.
\end{theorem}

We will prove many new results on CEDFs in Section \ref{CEDF.sec}, mainly concerning the case when $\lambda = 1$.
Our main results can be summarized as follows.

\begin{theorem} 
\label{main.thm}
\mbox{\quad} 
\begin{enumerate}
\item If $m$ is even and $\ell \geq 1$, then there exists an $(m \ell^2 + 1,m,\ell; 1)$-CEDF.
\item If $\ell$ and $m$ are both odd, then an $(m \ell^2 + 1,m,\ell; 1)$-CEDF does not exist.
\item If $\ell = 2$ and $4m+1$ is prime, then there exists an $(m \ell^2 + 1,m,\ell; 1)$-CEDF.
\end{enumerate}
\end{theorem}

We note that Theorem \ref{main.thm} is an immediate consequence of four theorems from Section \ref{CEDF.sec}, namely,
Theorems \ref{thm:CEDFm4}, \ref{thm:argh}, \ref{nonexist.thm} and \ref{ell=2.thm}.

\medskip

Now we proceed to the definitions for the related   ``strong'' versions of CEDF.

\begin{definition}[Strong circular AMD code]
Let $\mathcal{G}$ be an additive abelian group of order $n$ and $\mathcal{A} = \{ {A}_0,\dots, {A}_{m-1}\}$ be $m$ pairwise disjoint $\ell$-subsets of $\mathcal{G}$. Let $0 < \epsilon < 1$. Then $(\mathcal{G},\mathcal{A})$ is an  \emph{ $\epsilon$-secure strong circular $(n,m,\ell)$-AMD code} if an adversary cannot win the following \emph{strong circular AMD game} with probability greater than $\epsilon$.
\begin{enumerate}
\item The \emph{source} $i \in \{0,\dots,m-1\}$ is specified and given to the adversary.  \vspace{-.08in}
\item The adversary chooses a value $\Delta \in \mathcal{G} \setminus \{0\}$.   \vspace{-.08in}
\item The source is encoded by choosing $g$ uniformly at random from ${A}_i$. \vspace{-.08in}
\item The adversary wins if and only if $g + \Delta \in A_j$, where $j = i + 1 \bmod m$.  \vspace{-.12in}
\end{enumerate}
\end{definition}

\medskip

\begin{definition}[Strong circular external difference family (SCEDF)]
Let $G$ be an additive abelian group of order $n$. An $(n, m, \ell; \lambda)$-strong circular external difference family (or $(n, m, \ell; \lambda)$-SCEDF) is a set of $m$ disjoint $\ell$-subsets of $G$, say $\mathcal{A} = \{A_0,\dots,A_{m-1}\}$, such that the following multiset equation holds for every $j$, $0 \leq j \leq m-1$:
\[
\mathcal{D}(A_{j+1\bmod m}, A_j) = \lambda (G \setminus \{0\}).
\]
We observe that $\ell^2 = \lambda (n-1)$ if an $(n, m, \ell; \lambda)$-SCEDF exists.
\end{definition}

Similar to the case of weak circular $(n,m,\ell)$-AMD codes, an $\epsilon$-secure strong circular $(n,m,\ell)$-AMD code
has $\epsilon \geq {\ell}/({n-1})$. Again, if $\epsilon = {\ell}/({n-1})$, the code is \emph{R-optimal}.
Analogous to Theorem \ref{VSthm}, we have the following result that we state without proof.

\begin{theorem}
An R-optimal circular strong $(n,m,\ell)$-AMD code is equivalent to an $(n,m,\ell;\lambda)$-SCEDF.
\end{theorem}

SCEDFs were defined in \cite{VS}. As we mentioned in Section \ref{intro.sec}, SCEDF with $m > 2$ do not exist. In Section \ref{strong.sec}, we give some constructions for near-optimal strong circular $(n,m,\ell)$-AMD codes.

Finally, for future use, we recall some definitions of ``noncircular'' external difference families.

\begin{definition}
Let $G$ be an additive abelian group of order $n$. Suppose $m \geq 2$. 
An {\em $(n, m, \ell; \lambda)$-external difference family} (or {\em $(n, m, \ell; \lambda)$-EDF}) is a set of $m$ disjoint $\ell$-subsets of $G$, say $\mathcal{A} = (A_0,\dots,A_{m-1})$, such that the following multiset equation holds:
\[
\bigcup_{i}\bigcup_{j\neq i} \mathcal{D}(A_i, A_j) = \lambda (G \setminus \{0\}).
\]
\end{definition}

\begin{definition}[\cite{PS16}]
Let $G$ be an additive abelian group of order $n$. Suppose $m \geq 2$. 
An {\em $(n, m, \ell; \lambda)$-strong external difference family} (or {\em $(n, m, \ell; \lambda)$-SEDF}) is a set of $m$ disjoint $\ell$-subsets of $G$, say $\mathcal{A} = (A_0,\dots,A_{m-1})$, such that  the following multiset equation holds for each $i\in\{0,1,\dotsc,m-1\}$:
\[
\bigcup_{j\neq i} \mathcal{D}(A_i, A_j) = \lambda (G \setminus \{0\}).
\]
\end{definition}
\begin{definition}
An {\em $(n,m;\ell_1,\dotsc,\ell_m;\lambda_1,\dotsc,\lambda_m)$-GSEDF} (or {\em generalised strong EDF}) is a set of $\ell$ disjoint subsets $A_1,A_2,\dotsc,A_m$ of an abelian group $G$ of order $n$ such that $|A_i|=\ell_i$ for $1 \leq i \leq m$, and such that, for each $1 \leq i \leq m$, it holds that
\begin{align*}
\bigcup_{j \neq i} \mathcal{D}(A_i,A_j)=\lambda_i (G \setminus \{0\}).
\end{align*}

\end{definition}

\section{Novel Graph Labellings and CEDFs with $\lambda=1$}
\label{CEDF.sec}

If we are given some type of difference family or external difference family, we can associate with it a graph in which the vertices are identified with the elements of the sets of the difference family, and the edges correspond to pairs of elements whose differences we wish to consider.  In the case of a difference set, for example, the corresponding graph would be a complete graph, whereas for an $(n,m,\ell;1)$-external difference family it would be complete multipartite graph that has $m$ independent sets of size $\ell$.  When a (circular) external difference family exists in $\mathbb{Z}_n$, there are natural connections with {\em graceful labellings} (see \cite{ejcsurvey} for a comprehensive survey of graph labellings), which can potentially be exploited in the construction of (circular) external difference families.

\subsection{$\alpha$-valuations}\label{sec:alpha}
In this section we consider a class of graph labellings that do not directly yield CEDFs in general, but whose structure will inform the constructions we give in Section~\ref{sec:cedfcon}.
\begin{definition}[\cite{rosa}]
A {\em $\beta$-valuation} of a graph $G$ with $n$ edges is a one-to-one map of the vertices into the set of integers $\{0,1,\dotsc,n\}$, such that, if we label each edge by the absolute value of the differences of the labels of the corresponding vertices, then the resulting edge labels are precisely the elements of the set $\{1,2,\dotsc,n\}$.  (This is also known as a {\em graceful labelling}.)
\end{definition}
\begin{example}
Here is an example of a $\beta$-valuation for the graph $C_3$:
\begin{center}
\begin{tikzpicture}[main/.style = {draw, circle}] 
\node[main] (1) {$0$};
\node[main] (2) [above right of=1]{$1$};
\node[main] (3) [below right of=2]{$3$};
 \draw (1)--(2);
 \draw (3)--(2);
 \draw (1)--(3);
\end{tikzpicture}
\end{center}
\end{example}
\begin{definition}[\cite{rosa}]
An {\em $\alpha$-valuation} of a graph $G$ with $n$ edges is a $\beta$-valuation that satisfies the additional condition that there is there is some value $x$ with $0\leq x\leq n$ such that each edge is incident with one vertex whose label is at most $x$, and one whose label is greater than $x$.
\end{definition}
A graph with an $\alpha$-valuation is necessarily bipartite, with the partition given by the set $V^{\sf large}$, which consists  of vertices whose labels are larger than $x$, and $V^{\sf small}$, which consists of vertices whose labels are at most $x$.
\begin{example}\label{ex:C4}
Here is an example of an $\alpha$-valuation for the graph $C_4$, with $x=1$:
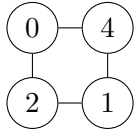
\begin{center}
\begin{tikzpicture}[main/.style = {draw, circle}] 
\node[main] (1) {$0$};
\node[main] (2) [right of=1]{$4$};
\node[main] (3) [below of=2]{$1$};
\node[main] (4) [below of=1]{$2$};
 \draw (1)--(2);
 \draw (3)--(2);
 \draw (4)--(3);
 \draw (4)--(1);
\end{tikzpicture}
\end{center}
We see that $V^{\sf small}$ contains the vertices with labels $\{0,1\}$ and $V^{\sf large}$ contains the vertices with labels $\{2,4\}$.
\end{example}


The following theorem illustrates a connection between graph labellings of this type and certain strong external difference families.

\begin{theorem}\label{thm:SEDF}
An $\alpha$-valuation of the complete bipartite graph $K_{\ell,\ell}$ implies the existence of an $(\ell^2+1,2,\ell,1)$-SEDF in $\mathbb{Z}_{\ell^2+1}$.
\end{theorem}
\begin{proof}
The graph $K_{\ell,\ell}$ has $\ell^2$ edges.  The edge labels come from taking the absolute value of the differences between their incident vertices.  In the case of a graph with an $\alpha$-valuation, the positive differences occur as $v-w$ for some label $v$ of a vertex in $V^{\sf large}$ and some label $w$ of a vertex in $V^{\sf small}$.  Since each value from $1$ up to $\ell^2$ occurs once as one of these differences, we see that $\mathcal{D}(V^{\sf large},V^{\sf small})=\mathbb{Z}_{\ell^2+1} \setminus \{0\}$. Therefore we also have $\mathcal{D}(V^{\sf small},V^{\sf large})=\mathbb{Z}_{\ell^2+1} \setminus \{0\}$, and the sets of labels of the vertices in $V^{\sf large}$ and $V^{\sf small}$ respectively form the two sets in the desired SEDF.
\end{proof}
Analogously, an $\alpha$-valuation of $K_{a,b}$ implies the existence of a $(K_{ab+1},a,b;1;1)$-GSEDF.

\medskip

The construction of SEDFs with $m=2$ and $\lambda=1$ given in \cite{PS16} can be viewed as an example of Theorem~\ref{thm:SEDF} being applied to an $\alpha$-valuation for complete bipartite graphs that appears in \cite{rosa}.  
\begin{remark}
{\rm 
Not every SEDF with $m=2$ arises from an $\alpha$-valuation in this way.  For example, consider the $(17,2,4,1)$-SEDF (from \cite{HucJefNep}) having sets $\{1,4,13,16\}$ and $\{2,8,9,15\}$ .  The elements of the second set interleave the elements of the first set. Therefore, even if we translate all elements by adding a common element of $\mathbb{Z}_{17}$, it is not possible to ensure that all elements of the first set are smaller than those of the second set when they are regarded as integers.
However, the elements in the first set are all one more than a multiple of 3.  If we subtract 1 from all elements of both sets, then multiply them all by 6 (the inverse of 3 mod 17) then we get the sets $\{0,1,4,5\},\{6,8,14,16\}$, which does come from an $\alpha$-valuation, i.e., this SEDF is \emph{equivalent} to one arising from an $\alpha$-valuation.
}
\end{remark}

\subsection{A Blow-up Construction}
We now consider a technique that allows us to take a graph with an $\alpha$-valuation and construct larger graphs that also have $\alpha$-valuations.
\begin{definition}
A {\em blow-up} of a graph $G$ is a graph constructed by replacing every vertex of $G$ with a finite collection of copies of that vertex, with any two copies of distinct vertices being joined by an edge if and only if the corresponding vertices were adjacent in $G$.
\end{definition}
\begin{theorem}\label{thm:blowup1}
Let $G$ be a graph with with $n$ edges that has an $\alpha$-valuation, and let $\ell$ be a positive integer. The blow-up obtained by replacing each vertex in $V^{\sf large}$ by a set of $\ell$ vertices has an $\alpha$-valuation.
\end{theorem}
\begin{proof}
Let $G^\prime$ be the labelled graph obtained by multiplying each of the labels on the vertices of $G$ by $\ell$.  The edge labels of this graph are precisely the multiples of $\ell$ from $1$ up to $n\ell$.  Now we create a labelled graph $G^{\prime\prime}$ by blowing up $G^{\prime}$: for every vertex of $V^{\sf large}$ with label $\ell i$, we replace it by $\ell$ copies with the labels $\{\ell i,\ell i-1,\dotsc,\ell i-(\ell-1)\}$.  By construction the vertex labels of $G^{\prime\prime}$ are all distinct.  The result of this blow-up is that each edge of $G^\prime$ is replaced by $\ell$ edges in $G^{\prime\prime}$.  If an edge of $G^\prime$ had label $j\ell$ then the new edges have labels $j\ell,j\ell-1,\dotsc,j\ell-(\ell-1)$.  Thus $G^{\prime\prime}$ contains $\ell n$ edges labelled with the values $1,2,\dotsc,\ell n$, so this labelling is a $\beta$-valuation. We observe that if $x$ is the largest label of any vertex of $V^{\sf small}$ in $G$ then $\ell x$ is the largest label  of any vertex of $V^{\sf small}$ in $G^{\prime\prime}$.  The label of each vertex of $V^{\sf large}$ in $G^\prime$ was at least $\ell(x+1)$.  Thus the labels of any copies in $G^{\prime\prime}$ are at least $\ell x+1$ by construction, and so this labelling is an $\alpha$-valuation.
\end{proof}
\begin{theorem}\label{thm:blowup2}
Let $G$ be a graph with with $n$ edges that has an $\alpha$-valuation, and let $\ell$ be a positive integer. The blow-up obtained by replacing each vertex in $V^{\sf small}$ by a set of $\ell$ vertices has an $\alpha$-valuation.
\end{theorem}
\begin{proof}
This proof is analogous to that of Theorem~\ref{thm:blowup1}, except that the labels of the copies of a vertex of $V^{\sf small}$ with label $\ell i$ are given the labels $\ell i,\ell i+1,\dotsc,\ell i+\ell-1$.
\end{proof}

If we apply the constructions of Theorem~\ref{thm:blowup1} and Theorem~\ref{thm:blowup2} in turn to a graph $G$, then the resulting graph is the lexicographic product $G \boldsymbol{\cdot} K_{\ell}^c$. Therefore we have the following theorem.

\begin{theorem}
\label{lex.thm}
Suppose a graph $G$ has an $\alpha$-valuation, and let $\ell \geq 2$. Then $G \boldsymbol{\cdot} K_{\ell}^c$ has an $\alpha$-valuation.
\end{theorem}

\begin{example}\label{ex:bipartitealpha}
Consider the bipartite graph $K_{1,1}$.  It has the following $\alpha$-valuation:
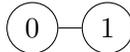
\begin{center}
\begin{tikzpicture}[main/.style = {draw, circle}] 
\node[main] (1) {$0$};
\node[main] (2) [right of=1]{$1$};
 \draw (1)--(2);
\end{tikzpicture}
\end{center}

If we apply the blow-up described in Theorem~\ref{thm:blowup1} with $\ell=3$, the result is the following $\alpha$-valuation of $K_{1,3}$:
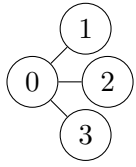
\begin{center}
\begin{tikzpicture}[main/.style = {draw, circle}] 
\node[main] (1) {$0$};
\node[main] (2) [above right of=1]{$1$};
\node[main] (3) [right of=1]{$2$};
\node[main] (4) [below right of=1]{$3$};
 \draw (1)--(2);
 \draw (3)--(1);
 \draw (1)--(4);
\end{tikzpicture}
\end{center}
If we then apply the blow-up described in Theorem~\ref{thm:blowup2} with $\ell=3$, we obtain an $\alpha$-valuation of $K_{3,3}$:
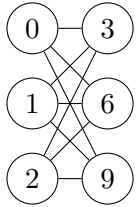
\begin{center}
\begin{tikzpicture}[main/.style = {draw, circle}] 
\node[main] (1) {$0$};
\node[main] (2) [below of=1]{$1$};
\node[main] (3) [below of=2]{$2$};
\node[main] (4) [right of=1]{$3$};
\node[main] (5) [right of=2]{$6$};
\node[main] (6) [right of=3]{$9$};
 \draw (1)--(4);
 \draw (1)--(5);
 \draw (1)--(6);
\draw (2)--(4);
 \draw (2)--(5);
 \draw (2)--(6);
 \draw (3)--(4);
 \draw (3)--(5);
 \draw (3)--(6);
\end{tikzpicture}
\end{center}
We observe that if we carry out this process with general $\ell$ we obtain an $\alpha$-valuation for $K_{\ell,\ell}$.  Applying Theorem~\ref{thm:SEDF} to this $\alpha$-valuation gives rise to the construction of an $(\ell^2+1,2,\ell;1)$-SEDF from \cite{PS16}.

In \cite{HucJefNep}, Huczynska, Jefferson and Nep\v{s}insk\'{a} describe a recursive construction for SEDFs and GSEDFs with $m=2$ and $\lambda=1$.  These can be seen as arising from performing a sequence of the blow-up operations described above. 
\end{example}
\subsection{CEDFs Arising From Graceful Digraphs}
\label{sec:cedfcon}
In \cite{BloomHsu82,BloomHsu85}, Bloom and Hsu consider graceful labellings for directed graphs (we note that in this paper, the digraphs we consider do not contain loops). They give the following definition:
\begin{definition}[\cite{BloomHsu82,BloomHsu85}]
A {\em graceful labelling} of a digraph $G$ with $n$ arcs is a mapping from the vertices of $G$ into the set $\{0,1,\dotsc,n\}$ such that, if we label each arc by the label on its head minus the label on its tail (modulo $n+1$), then each nonzero element of $\mathbb{Z}_{n+1}$ occurs exactly once as an arc label.  A digraph $G$ is a {\em graceful digraph} if it has a graceful labelling.
\end{definition}
\begin{lemma}[\cite{BloomHsu85}]
A $\beta$-valuation for a graph $G$ is also a graceful labelling for the digraph obtained from $G$ by directing each edge from the incident vertex of smaller label to the incident vertex of larger label.
\end{lemma}
\begin{lemma}\label{ex:alphadigraph}
An $\alpha$-valuation for a graph $G$ is also a graceful labelling for the digraph obtained by directing all edges of $G$ from the vertices of $V^{\sf small}$ to those of $V^{\sf large}$. 
\end{lemma}

\medskip

Now, consider the digraph obtained by orienting the edges of $C_m$ clockwise around the cycle; we refer to this as a {\em directed cycle}.  
The following result is an immediate consequence of the relevant definitions.
\begin{lemma}
A graceful directed cycle with $m$ arcs is equivalent  to an $(m+1,m,1;1)$-CEDF in $\mathbb{Z}_{m+1}$.
\end{lemma}
Bloom and Hsu show that a directed cycle  with $m$ arcs is never graceful if $m$ is odd.  This can be expressed as the following theorem:
\begin{theorem}
If there exists an $(m+1,m,1;1)$-CEDF in $\mathbb{Z}_{m+1}$, then $m$ is even.
\end{theorem}
\begin{proof}
Let a $(m+1,m,1,1)$-CEDF be given by $\{\{a_1\},\dotsc,\{a_m\}\}$.  
Denote \[
S = (a_1-a_m)+\sum_{i=2}^m (a_i-a_{i-1}).\]

As the elements of $\mathbb{Z}_{m+1}^*$ each occur once as a ``directed'' difference, we have 
\begin{align*}
S&\equiv \sum_{i=1}^{m} i \pmod{m+1}\\
&\equiv \frac{m(m+1)}{2} \pmod{m+1}.\end{align*}
{However, by construction,} 
$S \equiv 0 \pmod{m+1}.$
Thus \[ (m+1)\mid \frac{m(m+1)}{2},\] which does not hold if $m$ is odd.
\end{proof}
In \cite{BloomHsu85}, Bloom and Hsu mention that the directed cycle with $m$ arcs is graceful for all even $m$, and state that constructions from such labellings arise from known constructions for {\em complete mappings}.
\begin{definition}
A bijection $\theta$ of an additive group $G$ to itself is a {\em complete mapping} if the map $\sigma\colon G\rightarrow G$ given by $\sigma(x)=x+\theta(x)$ is also a bijection. In this case, $\sigma$ is known as an {\em orthomorphism} of $G$.
\end{definition}  
If $\mathbb{Z}_{m+1}$ has an orthomorphism $\sigma$, and if the permutation of the elements of $\mathbb{Z}_{m+1}$ induced by $\sigma$ consists a cycle of length $m$ together with a single fixed point, then labelling the vertices of the directed cycle with $n$ arcs by the elements \begin{align*}a,\sigma(a),\sigma^2(a),\dotsc,\sigma^{n-1}(a)\end{align*}
(where $a$ is not the fixed point) results in a graceful labelling.  To see this, we observe each arc has a label of the form $\sigma(x)-x$ for some $x$ where $\sigma(x)-x\neq 0$. We know that $\theta(x)=\sigma(x)-x$ is a bijection of $G$, so this cycle contains arc labels consisting of all nonzero elements of $G$.  Orthomorphisms of this form are directly related to the concept of {\em $R$-sequencing} of groups.
\begin{definition}[\cite{Evans}]
Let $G$ be an additive group of order $n$ with identity $0$. An {\em $R$-sequencing} of  $G$ is an ordering $a_0,a_1,\dotsc,a_{n-1}$ of the elements of $G$ such that 
\begin{itemize}
\item $a_0=0$,
\item the partial sums $a_0$, $a_0+a_1$, $a_0+a_1+a_2,\dotsc,a_0+a_1+\dotsb+a_{n-2}$ are distinct, and 
\item $a_0+a_1+\dotsb+a_{n-1}=0$.
\end{itemize}
\end{definition}
We observe that by definition, the list of partial sums arising from an $R$-sequencing of a group $G$ omits a single group element, $c$. It turns out that the map $\sigma$  defined by $\sigma(c)=c$, $\sigma(a_0)=a_0+a_1$, $\sigma(a_0+a_1)=a_0+a_1+a_2$, and so on up to $\sigma(a_0+a_1+\dotsb+a_{n-2})=e$, is an orthomorphism of $G$ (see \cite{Evans}).  We see therefore that an $R$-sequencing of $\mathbb{Z}_{m+1}$ gives rise to an $(m+1,m,1,1)$-CEDF.  These $R$-sequencings are known to exist when $m$ is even:
\begin{theorem}[\cite{Evans}]
For $m\equiv 0 \pmod{4}$, there is an $R$-sequencing of $\mathbb{Z}_{m+1}$ given by
\begin{multline*}
0,-1,2,-3,4,\dotsc,-\left(\frac{m}{2}-1\right),\frac{m}{2},\frac{m}{2}-1,-\left(\frac{m}{2}-2\right),\frac{m}{2}-3,-\left(\frac{m}{2}-4\right),\\ \dotsc,3,-2,1,-\frac{m}{2}.
\end{multline*}
For $m\equiv 2\pmod{4}$, there is an $R$-sequencing of $\mathbb{Z}_{m+1}$ given by
\begin{multline*}
0,-1,2,-3,4,\dotsc,-\left(\frac{m}{2}-2\right),\frac{m}{2}-1,-\left(\frac{m}{2}+1\right),\frac{m}{2}+2,-\left(m/n+3\right),\\ \dotsc,-\left(m-2\right),m-1,-m,\frac{m}{2}+1.
\end{multline*}
\end{theorem}
The CEDFs that we describe in Corollaries~\ref{cor:asdf} and \ref{thm:CEDFm4} can be viewed as arising from these $R$-sequencings.  (We choose to give explicit descriptions below in order to exploit a connection with $\alpha$-valuations.) Additional $R$-sequencings of $\mathbb{Z}_n$ are given in \cite{AAAP}; in turn these give rise to additional CEDFs with $\ell=1$.

In the case  $m\equiv 0\pmod{4}$, Rosa describes an $\alpha$-valuation for $C_m$ (see \cite{rosa}). We show here that this construction also gives rise to a graceful labelling of the directed cycle with $m$ arcs.
\begin{theorem}\label{thm:rosaalpha}  If $m\equiv0\pmod{4}$, then there is an $\alpha$-valuation of $C_m$.
\end{theorem}
\begin{proof}
 Denote the vertices of $C_m$ by the numbers $1,2,\dotsc,m$ in order round the cycle.  We label the vertices with elements of $\{0,1,2,\dotsc,m\}$ by giving vertex $i$ the label $a_i$ defined as follows:
\begin{align*}
a_i=\begin{cases}
(i-1)/2&\text{$i$ odd},\\
m+1-i/2&\text{$i$ even, $i\leq m/2$},\\
m-i/2&\text{$i$ even, $i>m/2$.}
\end{cases}
\end{align*}
According to Rosa, this labelling ``is evidently an $\alpha$-valuation" \cite{rosa}. 
\end{proof}
\begin{corollary}\label{cor:asdf}
If $m\equiv 0\pmod{4}$, then there is an $(m+1,m,1,1)$-CEDF.
\end{corollary}
\begin{proof}
As noted in Lemma~\ref{ex:alphadigraph}, we can think of the $\alpha$-valuation given in the proof of Theorem~\ref{thm:rosaalpha} as being a graceful labelling of the digraph in which, for each even $i\in\{1,2,\dotsc,m\}$, the arcs incident with vertices $i$ are directed towards $i$, since these vertices are precisely the elements of $V^{\sf large}$.  In order to convert this digraph into a directed cycle, we reverse the direction on those arcs of the form $(a_{i+1}, a_i)$ with $i$ even.  This results in negating (modulo $m+1$) the labels on these arcs.  We now consider the values on these labels:
\begin{itemize}
\item for $i$ even with $2\leq i\leq m/2$, the differences \[a_i-a_{i+1}=m+1-i/2-i/2=m+1-i\] take on all odd values from $m-1$ down to $m/2+1$;
\item for $i$ even with $m/2+2\leq i\leq m-2$, the differences \[a_i-a_{i+1}=m-i/2-i/2=m-i\] take on all even values from $m/2-2$ down to $2$;
\item finally, we have $a_m-a_1=m/2$.
\end{itemize}
This set of values is fixed under negation modulo $m+1$.  Thus, the labelled digraph obtained by reversing this set of arcs has the same set of arc labels as the $\alpha$-valuation digraph, and hence the corresponding labelling is graceful. 
\end{proof}
\begin{example}
An example of a $(13,12,1;1)$-CEDF in $\mathbb{Z}_{13}$ is given by \begin{align*}\{0\},\{12\},\{1\},\{11\},\{2\},\{10\},\{3\},\{8\},\{4\},\{7\},\{5\},\{6\}.
\end{align*}
\end{example}

We have constructed an $(m+1,m,1;1)$-CEDF for any $m\equiv 0 \pmod{4}$.  This approach can now be combined with the blowing-up technique to construct an $(m\ell^2+1,m,\ell;1)$-CEDF for any $\ell\geq 1$ and any $m\equiv 0 \pmod{4}$.
\begin{theorem}\label{thm:CEDFm4}
If $m\equiv0\pmod{4}$, then there is an $(m\ell^2+1,m,1;1)$-CEDF for any $\ell\geq 1$.
\end{theorem}
\begin{proof}
Suppose we have the $\alpha$-valuation of $C_m$ described in Theorem~\ref{thm:rosaalpha}.  
Applying Theorem \ref{lex.thm}, 
we obtain an $\alpha$-valuation of the lexicographic product $C_m \boldsymbol{\cdot} K_{\ell}^c$.
Suppose we consider the labels to be elements of $\mathbb{Z}_{m\ell^2+1}$. We show that this is a graceful labelling for the digraph obtained by orienting all edges clockwise around the cycle.

Where an edge between two vertices had a difference of $i$ in the initial $\alpha$-valuation of $C_m$, the corresponding independent sets, following the application of each blow-up, will be joined by $k^2$ edges that have the values $k^2i,k^2i-1,\dotsc,k^2(i-1)+1.$  As in the proof of Corollary~\ref{cor:asdf}, the set of arcs whose direction must be reversed to convert the $\alpha$-valuation digraph into the cyclically-ordered digraph is closed under negation, and so this labelling is graceful for both digraphs.
\end{proof}
\begin{example}
We can use  Corollary~\ref{cor:asdf} to construct a $(5,4,1;1)$-CEDF in $\mathbb{Z}_5$.  The resulting graph labelling is that depicted in Example~\ref{ex:C4}.  When we blow this up with $\ell=2$, as in the construction of Theorem~\ref{thm:CEDFm4}, we obtain the following labelled graph:
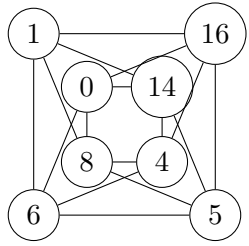
\begin{center}
\begin{tikzpicture}[main/.style = {draw, circle}] 
\node[main] (1) {$0$};
\node[main] (2) [right of=1]{$14$};
\node[main] (3) [below of=2]{$4$};
\node[main] (4) [below of=1]{$8$};
\node[main] (5) [above left of=1]{$1$};
\node[main] (6) [above right of=2]{$16$};
\node[main] (7) [below right of=3]{$5$};
\node[main] (8) [below left of=4]{$6$};
 \draw(5)--(2);
 \draw(5)--(6);
 \draw(1)--(2);
 \draw(1)--(6);
 \draw(2)--(3);
 \draw(2)--(7);
 \draw(6)--(3);
 \draw(6)--(7);
 \draw(3)--(4);
 \draw(3)--(8);
 \draw(7)--(4);
 \draw(7)--(8);
 \draw(4)--(5);
 \draw(4)--(1);
 \draw(8)--(1);
 \draw(8)--(5);
\end{tikzpicture}
\end{center}
Orienting the edges in a clockwise direction gives rise to the $(17,4,2;1)$-CEDF in $\mathbb{Z}_{17}$ whose sets are $\{0,1\}, \{14,16\},\{4,5\},\{8,6\}$.
\end{example}
For $m\equiv 2\pmod{4}$, there is no $\alpha$-valuation for $C_m$ (see \cite{rosa}).  However, we can adapt the construction of Theorem~\ref{thm:CEDFm4} to provide a construction of an $(m\ell^2+1,m,\ell;1)$-CEDF in this case also.  
\begin{theorem}\label{thm:argh}
If $m\equiv 2\pmod{4}$, then there is an $(m\ell^2+1,m,1;1)$-CEDF for any $\ell\geq 1$.
\end{theorem}
\begin{proof}
As before, we start by giving the vertices of $C_m$ labels $a_i$ for $i\in\{1,2,\dotsc,m\}$, this time defined as
\begin{align*}
a_i=\begin{cases}
(i-1)/2&\text{$i$ odd},\\
m+1-i/2&\text{$i$ even, $i\leq m/2-1$},\\
m-i/2&\text{$i$ even, $i\geq m/2+1$.}
\end{cases}
\end{align*}
We observe that this is not an $\alpha$-valuation, because the edge label $m/2$ occurs both as $|a_1-a_m|$ and as $|a_{m/2+1}-a_{m/2}|$.  However, every difference apart from $\pm m/2$ is obtained exactly once, and by a similar argument to that used in the proof of Corollary~\ref{cor:asdf} we can see that this gives rise to an $(m+1,m,1;1)$-CEDF.  Furthermore, we note that this labelling still has the property that each edge connects a vertex from a set $V^{\sf small}$, whose labels are contained in $\{0,1,\dotsc,m/2-1\}$, to a vertex from the set $V^{\sf large}$, whose labels are contained in $\{m/2,\dotsc,m\}$.  This makes it possible to apply the same blowing-up constructions described in Theorem~\ref{thm:CEDFm4}, and as a result we obtain a labelling that gives an $(m\ell^2+1,m,\ell;1)$-CEDF. 
\end{proof}
\begin{example}
An example of an $(11,10,1;1)$-CEDF in $\mathbb{Z}_{12}$ obtained by this method is given by \begin{align*}\{0\},\{10\},\{1\},\{9\},\{2\},\{7\},\{3\},\{6\},\{4\},\{5\}.
\end{align*}
\end{example}
We observe that the techniques used to prove that a directed cycle with $m$ arcs is not graceful for odd $m$ can be extended to show that an $(m\ell^2+1,m,\ell;1)$-CEDF does not exist if $\ell$ and $m$ are both odd.
\begin{theorem}
\label{nonexist.thm}
If $\ell$ and $m$ are odd then there is no $(m\ell^2+1,m,\ell;1)$-CEDF.
\end{theorem}
\begin{proof}
Consider the digraph obtained by blowing up each vertex of the cycle $C_m$ by a factor of $\ell$, then orienting all the edges clockwise around the cycle.  Suppose we have a graceful labelling of this digraph.  As each vertex has in-degree and out-degree both equal to $\ell$, this graph has a directed Eulerian trail.  If we sum (modulo $m\ell^2+1$), for each arc in turn around the trail, the label on the head of the arc minus the label on its tail, we get a result of $0$ since it is a closed trail.  However, if the labelling is graceful, then the edge labels are precisely the elements $\{1,2,\dotsc,m\ell^2\}$, and so this sum is equivalent to $m\ell^2(m\ell^2+1)/2
\pmod{m\ell^2+1}$, which is not equivalent to $0$ if $m$ and $\ell$ are both odd.  This gives a contradiction.
\end{proof}

It remains to consider the cases where $\ell$ is even and $m$ is odd. Several examples of CEDFs with $\ell$ even and $m$ odd can be obtained from Theorem \ref{thm2}, as shown in \cite{VS}. We also have the following new result for the case $\ell = 2$, which extends Theorem \ref{thm1}.

\begin{theorem}
\label{ell=2.thm}
Suppose that $q= 4m+1$ is a prime and $q > 5$. Then there is a $(q,m,2;1)$-CEDF in $\eff_q$.
\end{theorem}

\begin{proof}
Theorem \ref{thm1} states that the desired CEDF exists if there is a primitive element $\alpha \in {\eff_q}$ such that $\alpha^4-1$ is a quadratic non-residue. It was noted in \cite{VS} that the existence of such an $\alpha$ could be proven for all 
relevant primes and prime powers $q > 7.867 \times 10^8$. For all primes $5 < q <  10^9$ such that $q \equiv 1 \pmod{4}$, we verified computationally that there is an $\alpha$ satisfying the desired properties.
\end{proof}

\begin{remark}
{\rm For $1 \leq c \leq m-1$, Veitch and Stinson \cite{VS} defined $(q,m,\ell;\lambda)$-$c$-CEDF as a generalization of $(q,m,\ell;\lambda)$-CEDF. A $(q,m,\ell;\lambda)$-$1$-CEDF is the same thing as a $(q,m,\ell;\lambda)$-CEDF.
Wu, Yang and Feng prove in \cite[Corollary 1]{feng} that a $(q,m,2;1)$-$c$-CEDF can be constructed in $\eff_q$ for some positive integer $c \leq m-1$, provided that $q \equiv  1 \bmod 4$ is a prime or prime power and $q \geq 13$. Our Theorem \ref{ell=2.thm} proves the existence of $(q,m,2;1)$-$c$-CEDF with $c = 1$, but only for primes $q \equiv  1 \bmod 4$, $q \geq 13$. However, we expect that the condition on $\alpha$ used in the proof of Theorem \ref{ell=2.thm} could also be verified for prime powers, if desired.}
\end{remark}

\section{Strong Circular AMD Codes and SCEDFs}
\label{strong.sec}

In this section, we turn our attention to strong CEDFs and strong circular AMD codes. 
It is clear that any $(n,2,\ell;\lambda)$-SEDF is automatically an $(n,2,\ell;\lambda)$-SCEDF.
We also have the following general characterization of SCEDF in terms of SEDF, which follows immediately from the definitions.

\begin{theorem}
\label{T3}
Let $G$ be an additive abelian group of order $n$. A set of $m$ disjoint $\ell$-subsets of $G$, say $\mathcal{A} = \{A_0,\dots,A_{m-1}\}$ is an $(n, m, \ell; \lambda)$-SCEDF if and only if
$\{A_i,A_{i+1 \bmod m}\}$ is an $(n, 2, \ell; \lambda)$-SEDF for all $i$ such that $0 \leq i \leq m-1$.
\end{theorem}

Numerous constructions for $(n,2,\ell;\lambda)$-SEDF are known 
(for a recent summary, see Huczynska, Jefferson and Nep\v{s}insk\'{a} \cite{HucJefNep}), so we immediately obtain $(n,2,\ell;\lambda)$-SCEDF from them. 

Wu, Yang and Feng \cite{feng} used the group ring $\mathbb{Z}[G]$ to prove that there do not exist three disjoint sets $A_0$, $A_1$, and $A_2$ in an abelian group $G$ such that $\{A_0,A_1\}$ and $\{A_1,A_2\}$ are both SEDFs. It therefore follows immediately from Theorem \ref{T3} that SCEDFs with $m \geq 3$ do not exist. Thus, it is of interest to construct ``near-optimal''
strong circular AMD codes with $m > 2$, which can be done using classical results on cyclotomic numbers. We pursue this approach in the next section.

\subsection{Constructions Using Cyclotomy}

We now discuss how cyclotomic classes in finite fields can be used to construct near-optimal strong circular AMD codes.
The error probabilities of the AMD codes depend on known results concerning cyclotomic numbers.

Following \cite{BJWZ,Ding}, let $q = ef+1$ be prime  and let $\alpha \in \eff_q$ be a primitive element. Define
$C_0 = \{ \alpha^{je}: 0 \leq j \leq f-1\}$ and define $C_i = \alpha^i C_0$ for $1 \leq i \leq e-1$.
$C_0, \dots , C_{e-1}$ are the \emph{cyclotomic classes of index $e$}. We note that $\alpha^j C_i = C_{i+j \bmod e}$ for all $i,j$.

For a subset $S \subseteq \eff_q$ and $x \in \eff_q$, let $S + x = \{ s + x: s \in S\}$. 
The \emph{cyclotomic numbers of order $e$} are the integers denoted $(i,j)_e$ ($0 \leq i,j\leq e-1$) that are defined as follows:
\[ (i,j)_e = |  (C_i +1) \cap C_j  | .
\]

Obviously $\mathcal{A} = (C_0, \dots , C_{e-1} )$ can be viewed as a strong circular $(q,e,f)$-AMD code. The security of 
$\mathcal{A}$ depends on the cyclotomic numbers $(i,i+1 \bmod e)_e$, $0 \leq i \leq e-1$, as shown in the following theorem.

\begin{theorem}
\label{L1}
Let $\lambda = \max \{ (i,i+1 \bmod e)_e : 0 \leq i \leq e-1\}$. Then $\mathcal{A} = \{C_0, \dots , C_{e-1} \}$ is 
an $\epsilon$-secure  strong circular $(q,e,f)$-AMD code, where
$\epsilon = \lambda / f$. 
\end{theorem}

\begin{proof}
First, suppose that $i=0$ in the strong circular AMD game and the adversary chooses $\Delta \neq 0$. The adversary wins if $g + \Delta \in C_1$, where $g \in C_0$ is chosen uniformly at random. The probability $p_{\Delta}$ that the adversary wins is $|  \{ (C_0 + \Delta) \cap C_1\} | / f$. However,
\begin{eqnarray*} |(C_0 + \Delta) \cap C_1 | &=& |\Delta^{-1}(C_0 + \Delta) \cap \Delta^{-1}C_1 |\\
&=& |(\Delta^{-1}C_0 + 1) \cap \Delta^{-1}C_1| \\
&=& |(C_i + 1) \cap C_{i + 1 \bmod e}|, \quad\quad\text{where $\Delta^{-1} \in C_i$}.
\end{eqnarray*}
So \[p_{\Delta} = \frac{(i,i+1 \bmod e)_e }{ f},\] where $\Delta^{-1} \in C_i$. 
It follows that \[
\max \{ p_{\Delta} : \Delta \neq 0 \} = \frac{\max \{ (i,i+1 \bmod e)_e : 0 \leq i \leq e-1\}}{f} = \frac{\lambda}{f}.\]
If $i \neq 0$, the analysis is similar.
\end{proof}

There are $e$ cyclotomic numbers under consideration and their sum is $f$. It therefore follows immediately that $\lambda \geq f/e$ (where $\lambda$ is defined in Theorem \ref{L1}). Recall that $q = ef+1$. An R-optimal strong circular $(q,e,f)$-AMD would have $\epsilon = f/(q-1) = 1/e$.  We would obtain $\epsilon  = 1/e$ in Theorem \ref{L1} if and only if all the relevant cyclotomic numbers were equal to $f/e$. This can happen only if $e \mid f$, i.e., if $q \equiv 1 \bmod e^2$. However, we know of no situations with $e >3$ where equality actually occurs. 

\medskip

To illustrate this approach, suppose $q \equiv 1 \bmod 8$ is prime and we take $e=4$ in Theorem \ref{L1}. The security of the resulting strong circular $(q,4,(q-1)/4)$-AMD code
depends on the cyclotomic numbers $(0,1)_4$, $(1,2)_4$, $(2,3)_4$ and $(3,0)_4$.
The values of these cyclotomic numbers are computed as follows. First write $q = u^2 + 4v^2$, where $u \equiv 1 \bmod 4$. The value of $u$ is determined uniquely but the sign of $v$ is undetermined. Then (see, e.g., \cite{Ding}) we have
\begin{eqnarray*}
(0,1)_4 & = & \frac{q-3+2u + 8v}{16}\\
(1,2)_4 & = & \frac{q+1-2u}{16}\\
(2,3)_4 & = & \frac{q+1-2u}{16}\\
(3,0)_4 & = & \frac{q-3+2u - 8v}{16}.
\end{eqnarray*} 
Switching the sign of $v$ has the effect of interchanging the values of $(0,1)_4$ and $(3,0)_4$, but the resulting value of $\lambda$ is not affected.

\begin{example}
Suppose $q =17 = 4 \times 4 + 1$. We have $17 = 1^2 + 4 \times 2^2$, so $u = 1$ and $v = \pm 2$. 
The largest of the four cyclotomic numbers is \[\frac{17 - 3 + 2 + 16}{16} = 2.\] So we obtain a
$1/2$-secure strong circular $(17,4,4)$-AMD code from Theorem \ref{L1}. An R-optimal strong circular $(17,4,4)$-AMD code would have $\epsilon = 1/4$. 

Here are the details. If we start with the primitive element $3$, then 
\begin{eqnarray*}
C_0 &=& \{1,4,13,16\}\\
C_1 &=& \{3,5,12,14\}\\
C_2 &=& \{2,8,9,15\}\\
C_3 &=& \{6,7,10,11\}.
\end{eqnarray*}
Then we have $(0,1)_4 = 2$, $(1,2)_4 = 1$, $(2,3)_4 = 1$ and $(3,0)_4 = 0$.
 
$\mathcal{D}(C_1,C_0)$ contains two occurrences of each of $1,4,13$ and $16$ and one occurrence of each of $2,6,7,8,9,10,11$ and $15$.
Therefore, when $g \in C_0$, the adversary can choose $\Delta = 1,4,13$ or $16$ to win the strong circular AMD game with probability $1/2$. For example, $\Delta = 1$ results in a successful deception when the encoded source is $4$ or $13$. When $g \in C_1, C_2$ or $C_3$, we can determine the optimal choices for $\Delta$ by examining
$\mathcal{D}(C_2,C_1)$, $\mathcal{D}(C_3,C_2)$ or $\mathcal{D}(C_0,C_3)$ (resp.).
$\blacksquare$\end{example}

\begin{example}
Suppose $q = 97 = 4 \times 24 + 1$. We have $97 = 9^2 + 4\times 2^2$, so $u = 9$ and $v = \pm 2$. 
The largest of the four cyclotomic numbers is \[\frac{97 - 3 + 18 + 16}{16} = 8.\] So we obtain a
$1/3$-secure  strong circular $(97,4,24)$-AMD code from Theorem \ref{L1}. An R-optimal strong circular $(97,4,24)$-AMD code would have $\epsilon = 1/4$. $\blacksquare$
\end{example}

We can also analyse the asymptotic behaviour of this approach. Suppose we maximize the function
\[ \frac{q-3+2u + 8v}{16q/4} = \frac{q-3+2u + 8v}{4q}\]
subject to the constraint $q = u^2 + 4v^2$. Using elementary calculus, we see that
$  2u + 8v \leq 2\sqrt{5} \sqrt{q}$. Hence,
\[ \epsilon < \frac{q+ 2\sqrt{5} \sqrt{q}}{4q} = \frac{1}{4} + \frac{\sqrt{5}}{2}q^{-1/2}.\]
We have proven the following result.

\begin{theorem}
Suppose $q \equiv 1 \bmod 8$ is a prime power. Then there is an $\epsilon$-secure strong circular $(q,4,(q-1)/4)$-AMD code with $\epsilon <  \frac{1}{4} + \frac{\sqrt{5}}{2}q^{-1/2}$.
\end{theorem}

The cases where $q \equiv 5 \bmod 8$ is a prime can be handled analogously, except that the formulas for
the cyclotomic numbers are different.
As before, write $q = u^2 + 4v^2$, where $u \equiv 1 \bmod 4$ and the sign of $v$ is undetermined. Then we have
\begin{eqnarray*}
(0,1)_4 & = & \frac{q+1+2u - 8v}{16}\\
(1,2)_4 & = & \frac{q+1+2u + 8v}{16}\\
(2,3)_4 & = & \frac{q-3-2u}{16}\\
(3,0)_4 & = & \frac{q-3-2u}{16}.
\end{eqnarray*} 
We leave it to the reader to fill in the remaining details. 

\bigskip

For $e = 3$, we suppose $q \equiv 1 \bmod 6$ is a prime. The security of the resulting strong circular 
$(q,3,(q-1)/3)$-AMD code
depends on the cyclotomic numbers $(0,1)_3$, $(1,2)_3$ and $(2,0)_3$, which were originally determined by Gauss.
To compute these cyclotomic numbers, write $4q = u^2 + 27v^2$, where $u \equiv 1 \bmod 3$. Then we have
\begin{eqnarray*}
(0,1)_3 & = & \frac{2q-4 - u + 9v}{18}\\
(1,2)_3 & = & \frac{q+1+u}{18}\\
(2,0)_3 & = & \frac{2q-4 - u - 9v}{18}.
\end{eqnarray*} 
The sign of $v$ is again undetermined, but  $\lambda$  has the value 
\[  \frac{2q-4 - u + 9|v| }{18} .\]

\begin{example}
Suppose $q = 103$. We have \[4 \times 103 = 412 = 13^2 + 27\times 3^2,\] so $u = 13$ and $v = \pm 3$. 
The largest of the three cyclotomic numbers is \[\frac{2 \times 103 - 4 - 13 + 9 \times 3}{16} = 12.\] Since $f = 34$, we obtain a
$6/17$-secure strong circular $(103,3,34)$-AMD code from Theorem \ref{L1}. An optimal code (which does not exist) would be $1/3$-secure. $\blacksquare$
\end{example}

Various other values of $e$ can of course be considered in Theorem \ref{L1} and results can be obtained using known formulas for the relevant cyclotomic numbers.

\subsection{SEDFs and Tilings}
\label{tilings.sec}

In this section, we show that certain specific SEDF consisting of two sets cannot be ``extended'' to  SCEDF having $m \geq 3$ sets.
Of course this follows immediately as a consequence the general nonexistence result from \cite{feng}.  However, since the ``tiling'' technique we use in this section is a structural combinatorial proof (as opposed to the algebraic approach employed in \cite{feng}), we thought it might be of interest to briefly describe our results.

First, we consider the  ``standard'' $(\ell^2+1, 2, \ell; 1)$-SEDF in $\zed_{\ell^2+1}$ (see \cite{PS16}), consisting of the two sets 
$\{0, 1, \dots , \ell-1 \}$ and $\{\ell, 2\ell, \dots, \ell^2\}$.

Suppose $A_0 \subseteq G$, where $|G| = n$, $|A_0| = \ell$, and $n \equiv 1 \bmod \ell$. Denote $k = (n-1)/\ell$. A \emph{tiling} of $A_0$ is a set 
$T = \{ t_1, \dots , t_k\} \subseteq G$ such that
\[ \bigcup _{i = 1}^{k} (A_0 + t_i) = G \setminus \{0\}.\]
We refer to each set $A_0 + t_i$ as a \emph{translate} of $A_0$. 

\begin{lemma}
There exists an $(n, 2, \ell; 1)$-SEDF in $G$ if and only if there is a set $A_0 \subseteq G$ that has a tiling, where $|A_0| = \ell$ and $|G| = \ell^2+1$.
\end{lemma}

\begin{proof}
First, suppose $\{A_0,A_{1}\}$ is an $(\ell^2+1, 2, \ell; 1)$-SEDF in $\zed_{\ell^2+1}$. Let $T = -A_1$. Then $T$ is a tiling of $A_0$. Conversely, if $T$ is a tiling of $A_0$, it is easy to see that $\{A_0,-T\}$ is an $(\ell^2+1, 2, \ell; 1)$-SEDF.
\end{proof}

\begin{theorem}
\label{T2.2}
Suppose that $A_0 \subseteq G = \zed_{\ell^2+1}$ has a unique tiling, where $|A_0| = \ell$. Then the set $A_0$ cannot occur in any $(n, m, \ell; 1)$-SCEDF with $m \geq 3$.
\end{theorem}

\begin{proof}
Suppose $A_0, \dots , A_{m-1}$ is an $(n, m, \ell; 1)$-SCEDF, where $m \geq 3$.
Since the tiling of $A_0$ is unique, it follows that $A_1 = A_{m-1}$. But then $A_1$ and $A_{m-1}$ are not disjoint, which is not allowed in an SCEDF when $m \geq 3$.
\end{proof}

\begin{lemma}
\label{L2.3}
Suppose $G = \zed_{\ell^2+1}$. Then the set $A_0 = \{0, 1, \dots , \ell-1 \}$ has a unique tiling, namely, $T = \{1, \ell+1, \dots , \ell^2 - \ell +1\}$. 
\end{lemma}

\begin{proof} We need to find $\ell$ translates of $A_0 = \{0, 1, \dots , \ell-1 \}$ whose union (modulo $\ell^2+1$) is $\{1, \dots  , \ell^2\}$. Each translate of $A_0$ is an interval of $\ell$ consecutive residues modulo $\ell^2+1$. In order to cover all $\ell^2$ nonzero residues, the translates must be 
\begin{center}
$\{1, \dots , \ell\}$, $\{ \ell+1, \dots 2\ell\}$, $\dots$, $\{ \ell^2 - \ell + 1, \dots \ell^2\}$.
\end{center} Hence $T = \{1, \ell+1, \dots , \ell^2 - \ell +1\}$. 
\end{proof}

\begin{lemma} Suppose $A_0 \subseteq G = \zed_{\ell^2+1}$ has a  tiling $T$, where $|A_0| = \ell$.
Suppose $\gcd (c, \ell^2+1) = 1$ and suppose $d \in \zed_{\ell^2+1}$. Then $cT-d$ is a tiling of $cA_0 + d$.
\end{lemma}

\begin{proof}
We have \[ \bigcup _{i = 1}^{k} (A_0 + t_i) = G \setminus \{0\}.\]
Then \[ \bigcup _{i = 1}^{k} (cA_0 + d + c t_i- d) = \bigcup _{i = 1}^{k} c(A_0 + t_i) = c(G \setminus \{0\}) = G \setminus \{0\}.\]
\end{proof}

\begin{lemma}
\label{L2.5}
Suppose $G = \zed_{\ell^2+1}$. Then the set $A_0 = \{\ell, 2\ell, \dots, \ell^2\}$ has a unique tiling.
\end{lemma}

\begin{proof}
Let $c = -\ell \bmod (\ell^2+1)$ and let $d = -1$. Then $cA_0 + d = \{0, 1, \dots , \ell-1 \}$ (that is, $\{0, 1, \dots , \ell-1 \}$ and $\{\ell, 2\ell, \dots, \ell^2\}$ are equivalent under an affine transformation). We know from Lemma \ref{L2.3} that $\{0, 1, \dots , \ell-1 \}$ has a unique tiling, so $A_0$ also has a unique tiling. 
\end{proof}

\begin{corollary}
Suppose $m \geq 3$. No $(\ell^2+1, m, \ell; 1)$-SCEDF in $\zed_{\ell^2+1}$ can contain $\{0, 1, \dots , \ell-1 \}$ or $\{\ell, 2\ell, \dots, \ell^2\}$.
\end{corollary}

\begin{proof} This follows immediately from Theorem \ref{T2.2}, Lemma \ref{L2.3} and Lemma \ref{L2.5}.
\end{proof}

We can also use the machinery developed above to show that certain other $(\ell^2+1, 2, \ell; 1)$-SEDFs in $\zed_{\ell^2+1}$ cannot be extended. We consider a construction presented by Huczynska, Jefferson and Nep\v{s}insk\'{a} in \cite{HucJefNep}.

\begin{theorem}
\label{2a.thm}
\cite[Theorem 3.1]{HucJefNep}
Suppose $\ell = 2a$. Define
\[ A_0 = \{0, \dots , a-1\} \cup \{2a, \dots , 3a-1\}\] and 
\[ A_1 =  \{(4i-1)a, 4ia : i = 1, \dots , a\}.\]
Then $\{A_0,A_1\}$ is an $(\ell^2+1, 2, \ell; 1)$-SEDF in $\zed_{\ell^2+1}$.
\end{theorem}

\begin{lemma}
Suppose $G = \zed_{\ell^2+1}$ where $\ell = 2a$. Then the set $A_0$ defined in Theorem \ref{2a.thm} 
has a unique tiling.
\end{lemma}

\begin{proof} Each translate of $A_0$ consists of two disjoint sets of $a$ consecutive residues, separated by $a$ elements of $\zed_{\ell^2+1}$. It follows that if we have a translate $A_0 + t$ in $T$, we must also have exactly one of  the two translates $A_0 + t + a$ or $A_0 + t - a$. In either case, these two translates together cover $4a$ consecutive residues. There are in total $\ell^2 = 4a^2$ nonzero residues to cover, namely, $1, \dots , 4a^2$. So we have to choose translates $A_0+1$ and $A_0 + a+1$, which cover the first $4a$ nonzero residues. Then we must choose the translates $A_0 + 4a + 1$ and $A_0 + 5a+1$. Continuing, we see that
\[ T = \{ 4ia+1,(4i+1)a+1: 0 \leq i \leq a-1\}\] is the unique tiling of $A_0$.
\end{proof}

\begin{remark}
{\rm 
This tiling gives rise to the $(\ell^2+1, 2, \ell; 1)$-SEDF from Theorem \ref{2a.thm}, since
\begin{eqnarray*}
-T 
&=& \{ 4a^2-4ia, 4a^2-(4i+1)a : 0 \leq i \leq a-1\}\\
&=& \{ 4a(a-i), 4a(a-i) - a : 0 \leq i \leq a-1\}\\
&=& \{ 4aj, 4aj - a : 1 \leq j \leq a\} \quad \text{(setting $j = a-i$)}\\
&=& \{ 4aj, (4a-1)j  : 1 \leq j \leq a\} \\
&=& A_1.
\end{eqnarray*}
}
\end{remark}

\begin{lemma}
Suppose $G = \zed_{\ell^2+1}$ where $\ell = 2a$. Then the set $A_1$ defined in Theorem \ref{2a.thm} 
has a unique tiling.
\end{lemma}

\begin{proof}
The set $A_1 = \{3a,4a,7a,8a,11a,12a, \dots \}$ consists of $2a$ residues, each of which is divisible by $a$. Since there are in total $4a$ nonzero residues that are divisible by $a$, we require two translates of the form $A_1 + t_i$, where $t_i$ is divisible by $a$. We show that these two translates are uniquely determined. First, observe that $A_1$ consists of $a$ pairs of consecutive residues divisible by $a$, and each of these pairs is separated by two residues that are divisible by $a$. It follows that if we have a translate $A_0 + t$ in $T$, we must also have exactly one of the two translates $A_0 + t + 2a$ or $A_0 + t - 2a$.  In either case, these two translates together cover $4a$ consecutive residues divisible by $a$. We want these $4a$ residues to be $a, 2a, \dots, 4a^2$. The only way this can be done is to take $t_1 = 0$, $t_2 = -2a$.

In general, for any fixed $j$ such that $0 \leq j \leq a-1$, we require two translates $A_1+ t_i$ where $t_i \equiv j \bmod a$. Using a similar argument, these two translates are uniquely determined.
\end{proof}

\begin{remark} 
{\rm For the SEDF from Theorem \ref{2a.thm}, the sets $A_0$ and $A_1$ are not equivalent under an affine transformation
(except for $a=2$, where $A_1 = 2A_0 + 6$). Therefore the proof technique used in Lemma \ref{L2.5} cannot be applied,  so we had to prove that $A_0$ and $A_1$  have unique tilings using different arguments.
}
\end{remark}

\begin{corollary}
Suppose $m \geq 3$. No $(\ell^2+1, m, \ell; 1)$-SCEDF in $\zed_{\ell^2+1}$ can contain $A_0$ or $A_1$, where $A_0$ or $A_1$ are as defined in Theorem \ref{2a.thm}. 
\end{corollary}

\section{Summary}

There remain interesting open questions concerning CEDFs.
For example, are there examples of $(\ell^2 + 1,2, \ell; 1)$-SEDF that are not $\alpha$-valuations or equivalent to $\alpha$-valuations? The first open case is when $\ell = 5$. 
Also, we ask if there exists an $(m\ell^2+1,m,\ell;1)$-CEDF whenever $\ell \geq 2$ is even and $m \geq 3$ is odd. 
Finally, we note that CEDFs with $\lambda > 1$ have not received systematic study so far.

We used cyclotomy to find close approximations to certain (nonexistent) SCEDFs. However, this approach requires different calculations to be performed for every fixed value of $m$. It would therefore be of interest to find general constructions for good strong circular AMD codes.

\end{document}